 \documentclass{amsart}%
 \usepackage{amssymb}
 \usepackage{amsfonts}
 \usepackage{amsmath}
 \usepackage{graphicx}
 \usepackage{hyperref}%
 \usepackage{setspace}
 \usepackage{todonotes}
 \providecommand{\U}[1]{\protect\rule{.1in}{.1in}}
 \newtheorem{theorem}{Theorem}[section]
 \theoremstyle{plain}

 \newtheorem{corollary}{Corollary}[section]
 
 \newtheorem{definition}{Definition}[section]
 \newtheorem{example}{Example}[section]
 
 \newtheorem{lemma}{Lemma}[section]

 \newtheorem{proposition}{Proposition}[section]
 \newtheorem{remark}{Remark}[section]

 \numberwithin{equation}{section}

\begin{document}
 	\title[{\normalsize New Results On $S-r-$ideals in Commutative Rings}]{{\normalsize New Results On $S-r-$ideals in Commutative Rings }}
 	\author{Abuzer G\"und\"uz}
 	\address{Department of Mathematics, Faculty of Science, Sakaraya University, 
54100, Sakarya, T\"urkiye}
 	\email{abuzergunduz@sakarya.edu.tr}
 	\author{Osama A. Naji}
 	\address{Department of Mathematics, Faculty of Science, Sakaraya University, 
54100, Sakarya, T\"urkiye}
 	\email{osamanaji@sakarya.edu.tr}
    \author{Mehmet \"Ozen}
 	\address{Department of Mathematics, Faculty of Science, Sakaraya University, 
54100, Sakarya, T\"urkiye}
 	\email{ozen@sakarya.edu.tr}
 	\ 	\subjclass[2020]{13C05, 13B30, 13B25}
 	\keywords{$S-r-$ideal, $r-$ideal, Property A, prime ideal}
 	
 	\begin{abstract}
 	 This article studies the notion of  $S-r-$ideals in commutative ring $H$, where $S$ is a multiplicatively closed subset of $H$. Some basic properties of $S-r-$ideals are given. Various characterizations of $S-r-$ideals are presented. Also, $S-uz-$ring is defined and it is proved that $H$ is an $S-uz-$ring if and only if every maximal ideal disjoint from $S$ is an $S-r-$ideal provided $S$ is finite. In addition, the $S-r-$ideal concept is examined in amalgamation and trivial extension. Finally, $S-r-$ideals are studied in polynomial rings and it is investigated that when $A[x]$ is an $S-r-$ideal of $H[x].$

 	\end{abstract}
 	\maketitle
 	
 	\section{Introduction}\label{sec1}

Let $H$ be a commutative ring and $w\in H$. The set of elements annihilating $w$ is defined by $Ann(w)=\{y\in H: yw=0\}$. An element $w\in H$ is called regular (resp., zero divisor) if $Ann(w)=0$ (resp., $Ann(w)\neq 0)$. Consider the subset $T$ of $H$ and the ideal $A$ of $H$. The ideal $(A:T)$ is defined by $ \{w\in H: wT\subseteq A\}$ and we let $(0:T)=Ann(T)$. The sets $reg(H)$ and $zd(H)$ represent the set of all regular elements of $H$ and zero divisor elements of $H$, respectively. If $A$ contains a regular element, $i.e. \; A\cap r(H)\neq 0,$ then $A$ is said to be a regular ideal of $H$.

Let $S$ be a subset of $H$. If $1\in S$ and for each $s_1,s_2\in S$ we have $s_1\cdot s_2\in S,$ then $S$ is called a multiplicatively closed subset (briefly, m.c.s.). For an ideal $A$ of $H$, if $A\neq H$, then $A$ is called a proper ideal of $H.$ In many studies, $S-$versions of prime ideals and their generalizations have been researched such as \cite{HM}, \cite{Mass} and \cite{CH}.

    A proper ideal $A$ of $H$ is said to be $S-$prime if $S\cap A=\varnothing$ and there exists $s\in S$ such that for every $w,z\in H,\; wz\in A$ implies that $sw\in A$ or $sz\in A$. Note that $A$ is prime if $S=\{1\}$ \cite{HM}. The set $Min(A)$ is the set of all minimal prime ideals containing $A$. Also, all minimal prime ideals of $H$ can be shown by $Min(H).$ The set of all prime ideals (resp., maximal ideals) of $H$ is $Spec(H)$ (resp., $Max(H)$). An ideal $A$ of a reduced ring $H$ is called a $z^0-$ideal if $w\in A,\: z\in H$ with $Ann(w)=Ann(z)$, then $z\in A$. For more details about $z^0-$ideal see \cite{AM}, \cite{AKA} and \cite{Mas}.
    
Mohamadian introduced $r-$ideals of $H$ in \cite{M}. An ideal $A$ is called $r-ideal$  (resp., $pr-ideal$) of $H$ if for $w,z\in H,\; wz\in A$ and $Ann(w)=0$, then $z\in A (resp., \;z^n\in A$ for some $n\in \mathbb{N}).$ Mohamadian studied $r-ideals$ and defined $pr-ideals$, also investigated the relationship between $r-ideals$, prime ideals and $Min(A)$. He gave some results about polynomial ring and $r-ideals$. He also examined the connection between $z^0-$ideals and $r-$ideals. Leter, uniformly $pr-$ideals are studied in \cite{Uregen}.

Let $H$ be a ring, $S$  a m.c.s. and $A$ an ideal of $H$ disjoint from $S$. The ideal $A$ is called $S-r-ideal$ if $\exists s\in S$ such that for $w,z\in H,\; wz\in A$ and $Ann(w)=0$, then $sz\in A$ \cite{MH}.

 We know from \cite{M} that the prime ideals and $r-$ideals (and so $S-r-$ideals by taking $S=\{1_H\}$) are different. In fact, zero ideal in the ring $\mathbb{Z}_{6}$ is an $r-ideal$ but it is not prime, because $\bar{2}\cdot \bar{3}= \bar{0}\in \{\bar{0} \} $, however $2\notin \bar{0}$ and $\bar{3} \notin \{\bar{0}\}$. Also, not every prime ideal is an $r-ideal$. For instance, take $A=3\mathbb{Z}$ is prime in $\mathbb{Z}$, but it is not an $r-$ideal since $3\cdot 1=3\in A$, with $Ann(3)=0$, but $1\notin A$, see \cite{LC}.

    %Let $T$ be a subset of $R$. An element $a\in T$ is said to be von Neumann regular element if there exists $b\in T$ such that $a=a^2b.$ In this case $R$ is called von Neumann regular ring, namely all of their elements are von Neumann regular.\\

    A ring $H$ has Property A if every finitely generated ideal $B\subseteq zd(H)$ has non-zero annihilator. In addition, if for each finitely generated ideal $A$ of $H$ there is $z\in H$ with $Ann(A)=Ann(z)$, then $H$ satisfies the annihilator condition (briefly, a.c.).
    %and for each finitely generated ideal $I$ of $R$ there exists an element $b\in I$ with $Ann(I)=Ann(b)$, then $R$ satisfies strong annihilator condition (briefly s.a.c.).
    See \cite{Huck} for the more details and background on the above concepts. 
    
    Consider the ring $H$ and the $H-$module $E$. The idealization of $H$ (also called the trivial extension of $H$ by $E$), shown by $H\propto E$, was introduced by Nagata \cite[p.2]{Nagata}. $H\propto E$ is an abelian group with multiplication defined by $(w,e)(z,f)=(wz, wf+ze)$ where $w,z\in H$ and $e,f\in E$. For more information about trivial extensions, see \cite{Ander}, \cite{DKM}, \cite{KM}.

    Let $H_1$ and $H_2$ be commutative rings, $J$ be an ideal of $H_2$ and $f:H_1\rightarrow H_2$ be a ring homomorphism. The set $$H_1\bowtie^f J=\{(w,f(w)+j):\; w\in H_1, j\in J \} $$ is said to be the amalgamation of $H_1$ and $H_2$ along $J$ with respect to $f.$ 
    It is clear that it is a subring of $H_1\times H_2.$ 
    This construction generalizes the amalgamated duplication of a ring along an ideal which was introduced in \cite{DFF} by D'Anna and Fontana. In \cite{DFF2} and \cite{DF}, the more general context of amalgamations was introduced and studied in the frame of pullbacks allowing to give various results on the transfer of many properties from $A$ and $f(H_1)+J$ to $H_1\bowtie^f J.$
    The notion of amalgamation has gained great importance and received considerable attention during the last years. We refer to \cite{DFF}, \cite{DMZ}, \cite{MMZ2} and \cite{Gunduz} for more details.

 Inspired and motivated by the previous references and concepts, especially reference \cite{M},  we introduce this paper. The paper aims to study the concept of $S-r-ideal$. This study is composed of four sections. In section \ref{sec2}, we give some basic properties of $S-r-$ideals. We provide an example showing that not every $S-r-$ideal is an $r-$ideal (Example \ref{Ex.1}). We also introduce various characterizations of $S-r-$ideals. see Theorems \ref{Thm 2.1}, \ref{theorem1} and Proposition \ref{Prop 2.14}. Moreover, we define the $S-uz-$ring and show that $H$ is an $S-uz-$ ring if and only if every maximal ideal $M$ with $S\cap M=\varnothing$, is an $S-r-$ideal, provided $S\subseteq H$ is a finite m.c.s. (Proposition \ref{S-uz-prime}). In section \ref{sec3}; we examine the $S-r-$ideal concept in amalgamation and trivial extension. Finally, in section \ref{sec4}, we study $S-r-$ideals in polynomial rings. We investigate when $A[x]$ is an $S-r-$ideal of $H[x].$

 %In section \ref{sec2}, we present the definition of $S-r-$ideal and give some basic properties. We show that an $r-$ideal is $S-r-$ideal however the converse is not true by Example  \ref{Ex.1}. We also introduce various characterizations of $S-r-$ideal, see Theorems \ref{Thm 2.1}, \ref{Thm 2.2}, \ref{theorem1} and Propositions \ref{Prop 2.9} and \ref{Prop 2.14}. In section \ref{sec3}; we examine the $S-r-$ideal concept under ring homomorphism, Cartesian product, amalgamation and trivial extension. In section \ref{sec4}, we study $S-r-$ideals in polynomial rings. We investigate when $A[x]$ is an $S-r-$ideal of $H[x].$

 	\section{$S-r-$ideals and some Properties}\label{sec2}

\begin{definition}\cite{MH}
    Let $H$ be a ring, $S$ a m.c.s. and $A$ an ideal of $H$ with $A\cap S=\emptyset$. The ideal $A$ is called $S-r-ideal$ if $\exists s\in S$ such that for $w,z\in H,$ if $ wz\in A$ and $Ann(w)=0$, then $sz\in A.$
\end{definition}

The following example shows that not every $S-r-$ideal is an $r-$ideal.

\begin{example}\label{Ex.1}
    Let $H=\mathbb{Z}\times \mathbb{Z}$ and $S=u(\mathbb{Z})\times \mathbb{Z}$. Consider the ideal $A=0\times 2\mathbb{Z}$. Fix $s=(1,0)\in S$. Note from  [\cite{MH}, Theorem 3.1] and [\cite{M}, Proposition 2.8] that $A$ is not an $r-$ideal. However, $A$ is an $S-r$ideal. Indeed, $A\cap S=\varnothing,$ and let $(a,b)(c,d)\in A$ and $Ann(a,b)=(0,0)$. Observe that $ac=0$ and so $(c,0)(a,b)=(0,0)$. Hence $(c,0)\in Ann(a,b)=(0,0)$, thus $c=0$. Now, we have $s(c,d)=(1,0)(0,d)=(0,0)\in A$. Therefore $A$ is an $S-r-$ideal.
\end{example}

\begin{proposition} \label{Prop.23}
    Let $S_1 \subseteq S_2$ be two m.c.s.s of $H$. If $A$ is an $S_1-r-$ideal of $H$ and $A\cap S_2=\varnothing$ then $A$ is an $S_2-r-$ideal of $H$. The converse holds if $\forall s\in S_2, \; \exists r\in H$ such that $rs\in S_1.$
\end{proposition}

\begin{proof}
    Evident.
\end{proof}

\begin{example}\label{Ex.2}
Take an ideal $A$ of $\mathbb{Z}_{12}$.
Let $ \bar{w}, \bar{z} \in \mathbb{Z}_{12},\;  \bar{w}\cdot \bar{z} \in A$ and $Ann(\bar{w})=\bar{0}$. Then, $\bar{w}$ and $12$ are relatively prime since $Ann(\bar{w})=\bar{0}$. Thus $\bar{w}$ has an inverse in  $\mathbb{Z}_{12}$ and this implies that $\bar{z}=\bar{w}^{-1}\cdot(\bar{w}\bar{z}) \in A$. Namely, every ideal of $\mathbb{Z}_{12}$ is an $r-$ideal and hence it is an $S-r-$ideal for every multiplicatively closed set $S$ of $\mathbb{Z}_{12}$ with $S\cap A=\varnothing.$
\end{example}

Massoud and Hamid \cite{MH} showed that if $A$ is an $S-r-$ ideal of $H$ with $S\subseteq reg(H)$, then $S^{-1}A$ is an $r-ideal$. They also showed that the converse is true if $S^{-1}A\cap H=(A:s)$ for some $s\in S$. We show that the converse is true when $S$ is finite.

\begin{proposition}
    Let $A$ be an ideal of $H$ and $S \subseteq reg(H)$. If $S$ is finite and $S^{-1}A$ is an $r-$ideal, then $A$ is an $S-r-$ideal of $H$.
\end{proposition}

\begin{proof}
    Note that if $S\cap A\neq \varnothing$, then $S^{-1}A=S^{-1}H$, so we assume that $S\cap A=\varnothing$. Let $S=\{s_1, \cdots ,s_n\}$ be a finite set and put $s=s_1\cdot s_2 \cdots s_n$. Let $wz \in A$ for some $w,z \in H$ and $Ann(w)=0,$ then $\frac{w}{1} \cdot \frac{z}{1}\in S^{-1}A$. We show that $Ann(\frac{w}{1})=\frac{0}{1}=0$. Pick $\frac{r}{t} \in Ann(\frac{w}{1})$ and hence $\frac{r}{t} \cdot\frac{w}{1}=\frac{rw}{t}=\frac{0}{1}.$ Then for some $u \in S$. We have $ urw=0.$ As $S \subseteq reg(H),$ we get $rw=0 $ and so  $r \in Ann(w)=0$. Thus, $\frac{r}{t}=\frac{0}{1}.$ Since $S^{-1}A$ is an $r-ideal$, $\frac{z}{1}\in S^{-1}A$. So, $sz\in A.$
    
    %\Rightarrow \frac{r}{s}=\frac{0}{s}=\frac{0}{1}$ as desired.
\end{proof}

%The following result shows that a product of $S-r-$ideal with an ideal that is disjoint from $S$ is an $S-r-$ideal.

%\begin{proposition}
    %Let $H$ be a ring, $S$ a m.c.s, $A$ an ideal of $H$ and $B$ an ideal of $H$ such that $B\cap S \neq \varnothing$. If $A$ is an $S-r-$ideal of $H$, then $BA$ is an $S-r-$ideal of $H$.
%\end{proposition}

%\begin{proof}
    %Take $s\in S\cap B$ and let $t\in S$ be the element making $A$ satisfying $S-r-$ideal property. Fix $s^{'}=st\in S.$ Assume that $wz\in BA$ and $Ann(a)=0$ for some $w,z\in H.$ As $BA\subseteq A,$ we have $tz\in A.$ Hence $s^{'}z=(st)z=s(tz)\in BA.$
%\end{proof}

\begin{proposition}
    If $A\subseteq zd(H)$ is not an $S-r-$ideal, then there are two ideals $B$ and $K$ such that $B\cap reg(H)\neq \varnothing,\; A\subset_{\neq} B,K$ and $BK\subseteq A.$ 
\end{proposition}

\begin{proof}
    Assume that $\forall s\in S,\; \exists r,x \in H,\; rx\in A$ with $Ann(r)=0$ and $sx\notin A$. Since $rx \in A$, then $r\in (A:x)\subseteq (A:sx)$ and let $B=(A:sx)$ and $K=(A:B)$. It implies that $r\in B\setminus A,\; B\cap reg(H)\neq \varnothing$ and $BK\subseteq A$. Note that since $Ann(r)=0$ and  $r\notin zd(H)$, then  $r\notin A$. It is clear that $A\subseteq B,K.$
\end{proof}

In the following theorem, we give a characterization of $S-r-$ideal when $S$ is the set of all regular elements of $H.$

\begin{theorem}\label{Thm 2.1}
    Let $A$ be a proper ideal, $S=reg(H)$ and $S\cap A= \varnothing$. Then the following statements are equivalent:\\
    a) $A$ is an $S-r-$ideal.\\
    b) $\exists s \in S$ such that $s[rH\cap A]\subseteq rA,\; \forall r \in reg(H)$.\\
    c) $\exists s \in S$ such that $s(A:r) \subseteq A,\; \forall r \in reg(H)$.\\
    d) $s\pi^{-1}(S^{-1}A) \subseteq A$, where $\pi: H\rightarrow S^{-1}H$ is the natural map.
    
    %$sJ^c \subseteq I$ for some ideal $J$ in $Q(R)$.
\end{theorem}

\begin{proof}
    %We will show that the $a\Rightarrow b\Rightarrow c\Rightarrow a$, respectively.\\
    $(a)\Rightarrow (b):$ Let $s$ be the element satisfying $S-r-ideal$ property. Let $x \in rH\cap A.$ Then $x=rw, w\in H$ and $x\in A.$ So $rw \in A$ and since $r \in reg(H)$, then $Ann(r)=0$ and this implies that $sw \in A$ and so $swr \in rA$ and hence $sx\in rA$, as desired. As a result we get $s[rH\cap A]\subseteq rA$.\\
    $(b)\Rightarrow (c):$ Let $x \in (A:r)\Rightarrow xr \in A$ and at the same time $xr \in rH$ and so $xr\in rH\cap A.$ By $(b)$ we get $sxr \in rA$ and implies that $sxr=ry$ for some $ y\in A$. From $r(sx-y)=0$ and $Ann(r)=0$, we have $sx=y\in A$. Then $s(A:r)\subseteq A$. 
    %If $r\in I$ and it implies that $(I:r)=R$ and so $sR\subseteq I$ and so $s\in I\cap S$, a contradiction.\\
    $(c)\Rightarrow (d):$ Let $x\in \pi^{-1}(S^{-1}A).$ Then $\pi(x)=\frac{x}{1}\in S^{-1}A$. There exists $u\in S$ such that $ux\in A.$ This implies that $x\in (A:u)$ and by $c)$ we get $sx \in A$. That is, $s\pi^{-1}(S^{-1}A)\subseteq A$.\\
    $(d)\Rightarrow (a):$ Assume that $s\pi^{-1}(S^{-1}A) \subseteq A.$ Let $wz \in A$ and $Ann(w)=0$. Then $\pi(wz)=\frac{wz}{1}\in S^{-1}A$. As $w\in S$, we have $\frac{1}{w}\cdot \frac{wz}{1}=\frac{z}{1}=\pi(z)\in S^{-1}A.$ So $z\in \pi^{-1}(S^{-1}A)$ and hence by $(d)$ we obtain $sz\in A.$
\end{proof}

%In Theorem \ref{Thm 2.2}, we characterize $S-r-$ideals in terms of $r-$ideals. First, we present the following lemma.

 \begin{proposition}
\label{lemma:exponent}
Let $A$ be an $S-r-$ideal and $S\subseteq reg(H)$, then $\exists s\in S$ such that
$(A:s)=(A:s^n)$, for $n\geq 2.$
\end{proposition}

\begin{proof}
   Let $s\in S$ satisfy the $S-r-$ideal property. Let $x\in (A:s^n).$ Then $xs^n\in A$ and since $S\subseteq reg(H)$, we have $Ann(s^n)=0.$ This implies that $xs \in A$ and so $x\in (A:s).$ As a result, we get $(A:s)\supseteq (A:s^n)$. The converse is clear.
\end{proof}

Recall from [\cite{AKA}, Prop. 1.4] that an ideal $A$ of a reduced ring $H$ is called $z^0-$ideal if for $w,z\in H$ such that $w\in A$ and $Ann(w)=Ann(z)$, then $z\in A.$ Now we define $S-z^0-$ideal in reduced rings. Let $H$ be a reduced ring and $S\subseteq H$ a m.c.s.. An ideal $A$ of $H$ is said to be $S-z^o-ideal$ if $A\cap S=\varnothing$ and $\exists s\in S$ such that for $w,z\in H$ with $Ann(w)=Ann(z)$ and $w\in A$, then $sz\in A$.

%\begin{definition}
    %$A$ is $S-z^o-ideal$ if $\exists s \in S $ and if $A\cap S=\varnothing$ and $ w\in A$ and $z\in H$ with $Ann(w)=Ann(z)$, then $ sz \in A.$
%\end{definition}

\begin{proposition}
    If $A$ is an $S-z^o-ideal$ of a reduced ring $H$, then $A$ is an $S-r-ideal$ of $H$.
\end{proposition}

\begin{proof}
    Let $s\in S$ satisfy the $S-z^o-ideal$ property for $A.$ Let $wz\in A$ with $Ann(w)=0$ for some $w,z\in H.$ It is clear that $Ann(z)\subseteq Ann(wz)$. Let $r\in Ann(wz)$, then $rwz=(rz)w=0$. It implies that $rz\in Ann(w)=0$ and so $r\in Ann(z)$. Thus $Ann(z)=Ann(wz)$. Note that $(wz)z\in A$ and $wz \in A$ and $z\in H$, it gives $sz\in A$, as desired.
\end{proof}

%In the following theorem we give a characterization of $S-r-$ideal when $S$ is the set of all regular elements of $R.$

\begin{theorem}
    Suppose that $H$ is a ring and $L \in Min(A)$, where $A$ is an $S-r-ideal$. Then $L$ is an $S-r-ideal$ if $L\cap S=\varnothing. $
\end{theorem}

\begin{proof}
    Assume that $L\cap S=\varnothing$ and $ wz \in L$ with $Ann(w)=0$ for some $w,z\in H$. By \cite[Theorem 1.2]{Huck}, there exists $x \notin L$ and $n\in \mathbb{N}$ such that $x(wz)^n=xw^nz^n \in A.$ Since $Ann(w)=0$ and $A$ is an $S-r-ideal$, $s^nxz^n\in A\subseteq L$ for some $s\in S.$ As $x\notin L$, we obtain $s^nz^n\in L$ and so $(sz)^n \in L$. Hence $sz\in L,$ as desired.
    
    %$s^*(sb) \in P$ and implies that $s \in P$ or $b \in P$. But since $P \in Min(I)$ we have $s \notin P$ and so $(s^* s)b \in P$. For $t=s^*s$ and we get $tb \in P$, as desired. 
\end{proof}

\begin{theorem} \label{theorem1}
    Suppose that $H$ is a ring and $A$ is a prime ideal of $H$ with $A\cap S=\varnothing.$ Then, $A$ is an $S-r-$ideal if and only if $A\subseteq zd(H)$.
\end{theorem}

\begin{proof}
   $\Rightarrow:$ Suppose that $A\nsubseteq zd(H)$, then $\exists w\in A$ with $Ann(w)=0$. So we have $w\cdot 1 \in A$ and this implies that $s\cdot 1=s\in A$, a contradiction.\\
   $\Leftarrow:$ Let $wz \in A$ and $Ann(w)=0$, then $w$ is a regular element and by hypothesis $w\notin A$. Since $A$ is prime, we get $z\in A$ and so $sz\in A$, as desired.
\end{proof}

Note that Mohamadian [\cite{M}, Remark 2.3.(f)] showed that a prime ideal is an $r-$ideal if and only if it consists entriely of zero-divisors. Thus, we conclude that if $A$ is a prime ideal with $A\cap S=\varnothing$, then $A$ is an $r-$ideal if and only if $A$ is an $S-r-$ideal.

\begin{remark} \label{Remark1}
   Note that as $A\subseteq zd(H)$, by \cite[Chapter 1, Exercise 14]{Atiyah} , we can find a prime ideal $P$  such that $A\subseteq P \subseteq zd(H)$. By Theorem $\ref{theorem1}$ we get $P$ is an $S-r-$ideal if $P\cap S=\varnothing.$
\end{remark}

\begin{corollary}
    If $A=zd(H)$ is a prime ideal and $A\cap S=\varnothing$, then $A$ is an $S-r-ideal$ of $H.$
\end{corollary}

\begin{proof}
Follows from Theorem \ref{theorem1}.
    %Let $I$ be a prime ideal and $I=zd(R)$ and let $ab\in I$ with $Ann(a)=0$ for some $a,b\in R.$ Since $Ann(a)=0$, then $a\notin I$ and since $I$ is a prime ideal, then $b\in I$. Hence $I$ is an $r-ideal$ and so is an $S-r-ideal.$
\end{proof}

The following corollary is presented in [\cite{MH}, Proposition 2.3], also it can be derived from Theorem \ref{theorem1}.

\begin{corollary}
  If $A$ is an $S-r-ideal$, then $A\subseteq zd(H).$   
\end{corollary}

\begin{proof}
    
    Follows from Theorem \ref{theorem1} and the conditions $A$ is prime with $A\cap S=\varnothing$ is not needed for this directions.
    %Assume that $I \nsubseteq zd(R)$. By definition we have $\exists a \in I$ with $Ann(a)=0$ and since $a\cdot 1 \in I,\;\exists s\in S$ such that $s\cdot 1=s \in I$, a contradiction!
\end{proof}

The following proposition characterizes $r-$ideals in terms of $S-r$ideals.

\begin{proposition}\label{Prop 2.14}
    Let $A\subseteq J(H)$, where $J(H)$ is the Jacopson radical of $H.$ Then the following statements are equivalent:\\
    1) $A$ is an $r-ideal$ of $H$.\\
    2) $A$ is an $(H\setminus M)-r-$ideal for every maximal ideal $M$ of $H.$
\end{proposition}

\begin{proof}
    $(1) \Rightarrow (2):$ It is enough to note that as $A\subseteq J(H)$, we have $A\subseteq M,$ and so $A\cap (H\setminus M)=\varnothing$ for every maximal ideal $M.$ Since every $r-ideal$ is an $(H\setminus M)-r-$ideal, the result follows.\\
    $(2)\Rightarrow (1):$ Assume that $wz \in A$ with $Ann(w)=0$ for some $w,z\in H$. By hypothesis, $\forall M$ maximal ideal of $H,\; \exists s_m \in (H\setminus M)$ such that $s_m \cdot z \in A.$ Now, let the set $B$ be $\{s_m:\; \exists\; maximal\;ideal\;M,\;s_m\notin M\;and \;s_m z\in A \}$. We claim that $<B>=H.$ Assume that $<B>\neq H$, then $\exists\;maximal$ ideal $M^*$ such that $<B>\subseteq M^*$ with   $s^*\in <B>\subseteq M^* $ for some $ s^* \notin M^*$, a contradiction. Now, we use the claim and complete the proof. As $1\in <B>$, this implies that $1=r_1s_{m_1}+\cdots+r_ks_{m_k}$ and we can write $$z=z\cdot \sum_{i=1}^kr_is_{m_i}=\sum_{i=1}^kr_i(zs_{m_i})$$ and so $z\in A$ since $zs_{m_i}\in A$, as desired.
\end{proof}

%\section{New Results}

\begin{proposition}\label{zero.ideal}
    The zero ideal is always an $S-r-$ideal.
\end{proposition}

\begin{proof}
    Let $H$ be a ring, $S$ be a m.c.s. of $R$. For $w,z\in H,\;\; wz\in (0)$  and $Ann(w)=0.$ Since $wz=0$ and so $z\in Ann(w)$. Hence $z=0$ and so for every $s\in S$, $sz\in (0)$.
\end{proof}

\begin{proposition}\label{Ann(S)}
    If $A$ is an $S-r-$ideal, $K\subseteq H$ and $K\nsubseteq A$, in this case if $(A:K)\cap S=\varnothing,$ then $(A:K)$ is an $S-r-$ideal . Moreover, if $Ann(K)\cap S=\varnothing$, then $Ann(K)$ is always an $S-r-$ideal.
\end{proposition}

\begin{proof}
Let $wz\in (A:K)$ and $Ann(w)=0$. Then $wzK\subseteq A$. By hypothesis and [\cite{MH},Theorem 2.1] we have $szK\in A$ and so $sz\in (A:K)$, as desired. Especially if we take $A=(0)$, then $Ann(K)$ is always an $S-r-$ideal by Proposition \ref{zero.ideal}.
\end{proof}

\begin{proposition}
    Take $H$ to be a ring such that every ideal $A$ of $H$ is an annihilator ideal (i.e. for every ideal of $A$ there exists $T \subseteq R$ such that $A=Ann(T)$) with $Ann(T)\cap S=\varnothing$ then every ideal of $H$ is an $S-r-$ideal. Moreover, for any two ideals $A$ and $B$ in the ring $H$, there exists an ideal $K$ such that $Ann(A)+Ann(B)=Ann(K)$, in this case $Ann(A)+Ann(B)$ is an $S-r-$ideal.
\end{proposition}

\begin{proof}
    Follow by Proposition \ref{Ann(S)}.
\end{proof}

\begin{proposition}
    Assume that $K_1$ and $K_2$ are two subsets of $H$ such that $\exists t\in S$ with $K_1+K_2=Ht$. Let $K=Ann(K_1)+Ann(K_2)$. If $K\cap S=\varnothing$. then $K$ is an $S-r-$ideal.
\end{proposition}

\begin{proof} \label{Ann(A)}
Let $wz\in K$ such that $Ann(w)=0$. In this case there exists $k_1\in Ann(K_1)$ and $k_2\in Ann(K_2)$ such that $wz=k_{1}+k_{2}$, $k_1K_1=0$ and $k_2K_2=0$. So, $wzK_1K_2=(k_1+k_2)K_1K_2=k_1K_1K_2+k_2K_1K_2=0$ and hence $zK_1K_2\subseteq Ann(w)$ and thereby $zK_1K_2 = 0.$ As a result, $zHt=z(K_1+K_2)=zK_1+zK_2\subseteq Ann(K_1)+Ann(K_2)=K$ and hence $zt\in K,$ as desired.
\end{proof}

\begin{corollary}
 Take $H$ be a ring and $w,z\in H$ such that $w+z=t$, where $t\in S$. If $A=Ann(w)+Ann(z)$  with $A\cap S=\varnothing$, then $A$ is an $S-r-$ideal.
\end{corollary}

\begin{proof}
    Follow by Proposition \ref{Ann(A)}.
\end{proof}

\begin{proposition}
    
     Let $H$ be a reduced ring, $P\in Min(H)$ and $e\in H$ be an idempotent element. If $A=P+Ann(se)$ with $A\cap S=\varnothing$ and $s\in S$, then $A$ is an $S-r-$ideal.
\end{proposition}

\begin{proof}

     Take $rx\in A$ such that $Ann(r)=0$ and $x\in H.$ Therefore, $rx=w+z$, where $w\in P$ and $sze=0$. Obviously, there exists $t \notin P$ such that $swt=0$ by \cite[Corollar 2.2]{Huck}. Thereby, $setrx=0$, we obtain $setx=0$ and so $sex\in P.$ Then $sx=sex+(1-e)sx\in P+Ann(se)$, and so $A$ is an $S-r-$ideal.
    
\end{proof}

Let $f:H\rightarrow H'$ be a ring map and $S\subseteq H$ be m.c. subset an element $a\in H'$ is said to be $S-$idempotent if $a^2=f(s)a$ for some $s\in S$ [\cite{ETY}. Def. 2.1]. So, every idempotent element is an $S-$idempotent [\cite{ETY}. Def. 2.1].

\begin{proposition}
Assume that $s=\{s_1,\cdots,s_m \}$ is a m.c. subset of $H$ and $s=s_1\cdots s_m$. Take $T=\{a_i\in H; \; a_i^2=sa_i \; and \; i\in \Delta\}$. In this case, if $A\cap S=\varnothing$, then $A=\sum_{i\in A}a_iH$ is an $S-r-$ideal.
\end{proposition}

\begin{proof}
Take $wz\in A$ and $Ann(w)=0$ for some $w,z \in H$. We aim to show that $sz\in A.$ Write $wz=\sum^{n}_{k=1}a_{i_k}r_{i_k}$ for some $a_{i_1},\cdots, a_{i_n} \in T$ and $r_{i_1},\cdots, r_{i_n} \in H$. Set $y=\prod_{k=1}^n(s-a_{i_k})$. Hence, $wzy=0,$ and so $zy=0.$ Observe that $y$ can be written as $y=s^n-t$ for some $t\in A$. Thus, $z(s^n-t)=0$ and $s^n z=tz\in A$. As a result, $sz\in A,$ as desired.
    
\end{proof}

Recall from \cite{M} that a ring $H$ is said to be a $uz-$ring if each element if $H$ is either a zero divisor or a unit. for instance, $S^{-1}H$ is a $uz-$ring, where $S=reg(R)$, and Artinian rings are $uz-$rings \cite{M}.

\begin{definition}
Let $S\subseteq H$ be a m.c.s.. An element $a$ of $H$ is said to be $S-$unit if $Ha\cap S\neq \varnothing.$
\end{definition}

\begin{definition}
    A ring $H$ is said to be $S-uz-$ring if each element of $H$ is either a zero divisor or $S-unit$. 
\end{definition}

\begin{proposition}\label{S-uz}
Let $S=\{s_1,\cdots,s_n \}$ be a m.c.s. of $H$ and $s=s_1\cdot s_2\cdots s_n \in S.$ Then every ideal $A$ of $H$ with $A\cap S= \varnothing$ is an $S-r-$ideal if and only if $H$ is an $S-uz-$ring.
\end{proposition}

\begin{proof}

$(\Leftarrow):$ Let $A$ be an ideal with $A\cap S= \varnothing$ and $ab\in A$ with $Ann(a)=0$. Then $a$ is an $S-unit$, that is there exists $v\in H$ such that $av=s_r$ for some $s_r\in S$. Hence $sb=(s_1\cdots s_r\cdots s_n)b=(s_1\cdots s_{r-1} s_{r+1}\cdots s_n) (s_r b)=(s_1\cdots s_{r-1} s_{r+1}\cdots s_n) (va b)\in A$.\\

$(\Rightarrow):$ Let $a\in H$ and $a$ is a zero divisor, then we are done. Assume that $a$ is regular and $a$ is not an $S$-unit. Let $A=Ha$ with $A\cap S=\varnothing.$ As $A$ is an $S-r-$ideal and $a\cdot 1\in A$ and $Ann(a)=0$, then $s\cdot1\in A$ for some $s\in S$, a contradiction. Thus $a$ is either an $S-unit$ or a zero divisor.
\end{proof}

More generally, we have the following result.

\begin{proposition}\label{S-uz-prime}

   Let $S\subseteq H$ be a finite m.c.s. and $S\cap M=\varnothing$ for every maximal ideal $M$ of $H$. Then the following statements are equivalent.\\
    a) $H$ is an $S-uz-$ring.\\
    b) Every prime ideal $A$ of $H$ with $A\cap S=\varnothing$ is an $S-r-$ideal.\\
    c) Every maximal ideal of $H$ is an $S-r-$ideal.
\end{proposition}

\begin{proof}

$(a) \Rightarrow (b):$ Follow by Proposition \ref{S-uz}.\\
$(b) \Rightarrow (c):$ It is clear.\\
$(c) \Rightarrow (a):$ Let $a\in H$ be none $S-unit$ in $H$. Then $av\neq s,\; \forall v\in H$ and $\forall s\in S$. In particular, $av\neq 1$ for all $v\in H.$ Thus $a\in M$ for some maximal ideal $M$ of $H$. By $(c)$, $ M$ is an $S-r-$ideal, and by Theorem \ref{theorem1}, we obtain $M\subseteq zd(H)$, that is, $a\in zd(H)$. Namely, if $a$ is not an $S-unit$, then it is a zero divisor. Thus $H$ is an $S-uz-$ ring.
\end{proof}

\begin{example}
Let $F$ be a finite field and $M=<x-1>\subseteq F[x]$ is a maximal ideal. Then by [\cite{Kaplan}, Theorem 150] we can find a regular element $f\in M.$ Let $S=F\setminus \{0\}$.
 Note that $S\cap M= \varnothing$. Since $f\cdot 1\in M$ and $Ann(f)=0$, but $\forall s\in S,\; s\in S,\; s\cdot 1 \notin M.$ Then $M$ is not $S-r-$ideal of $F[x]$. Therefore , by Proposition \ref{S-uz-prime}, $F(x)$ can not be $S-uz-$ring.
\end{example}

\section{$S-r-$ideals in Some Special Rings}\label{sec3}

\begin{lemma}\label{isomorphism}
    If $f:H_1 \rightarrow H_2$ is an isomorphism, then $f(Ann(w))=Ann(f(w))$ for $w\in H_1$.
\end{lemma}

\begin{proof}
Let $y\in f(Ann(w))$. Then $\exists x\in Ann(w)$ such that $y=f(x)\in f(Ann(w)).$ Then $xw=0$ and so $f(xw)=f(x)f(w)=0$, which implies $y=f(x)\in Ann(f(w))$. Conversely, let $y\in Ann(f(w)).$ As $f$ is onto, $\exists x\in H_1$ such that $y=f(x)$. So, $f(x)\in Ann(f(w)).$ This means $f(x)f(w)=f(xw)=0.$ Since $f$ is $one-to-one$, we obtain $xw=0,$ that is, $x\in Ann(w)$. Thus, $y=f(x)\in f(Ann(w))$.
\end{proof}

%\begin{proposition}
    %Let $f:H_1 \rightarrow  H_2$ be a ring isomorphism, then the following are held:\\
    %1) If $B$ is an $f(S)-r-$ideal of $H_2$, then $f^{-1}(B)$ is an $S-r-$ideal of $H_1$.\\
    %2) If $A$ is an $S-r-$ideal of $H_1$, then $f(A)$ is an $f(S)-r-$ideal of $H_2$.
%\end{proposition}

%\begin{proof}
   % Evident.
%\end{proof}

%\begin{proposition}
    %If $A$ and $B$ are an $S_i -r-$ideals of $H_i$, where $i=1,2$, then $A\times B$ is an $S_1 \times S_2- r-ideal$ of $H_1 \times H_2$.
%\end{proposition}

%\begin{proof}
    %Let $(r_1,r_2), (x_1,x_2)\in H_1\times H_2$ such that $(r_1,r_2)\cdot (x_1,x_2) \in A\times B$ with $Ann(r_1,r_2)=(0,0)$. Then $(r_1x_1,r_2x_2)\in A\times B$. Let $w\in Ann(r_1)$. Then $(w,0)(r_1,r_2)=(wr_1,0)=(0,0)$ and hence $(w,0)\in Ann(r_1,r_2)=(0,0)$. That is, $w=0$ and $Ann(r_1)=0$. Similarly, $Ann(r_2)=0$. On the other hand, As $r_1x_1\in A$ and $r_2x_2\in B,\; \exists s_1 \in S_1$ and $\exists \; s_2 \in S_2$ such that $s_1x_1\in A$ and $s_2x_2 \in B$ and so $(s_1,s_2)\cdot (x_1,x_2)\in A\times B$, as desired.
%\end{proof}

\begin{proposition}
 
1)  Suppose that $f$ is an isomorphism. If $A \bowtie^f J$ is an $S\bowtie^f J-r-ideal$ of $H_1\bowtie^f J$, then $A$ is an $S-r-ideal$ of $H_1$.\\
2) Assume that $f$ is an epimorphism, $H_1$ is an integral domain and $J\subseteq zd(H_2)$. If $A$ is an $S-r-ideal$ of $H_1$, then $A \bowtie^f J$ is an $S\bowtie^f J-r-ideal$ of $H_1 \bowtie^f J$ 
\end{proposition}

\begin{proof}

1) Observe that $S\cap A=\varnothing.$ Let $wz\in A$ with $Ann(w)=0$, for some $w,z\in H_1.$ Then $(w,f(w))(z,f(z))=(wz,f(wz)) \in A \bowtie^f J$. We show that $Ann(w,f(w))=(0,0)$. Let $(x,f(x)+j)\in Ann(w,f(w))$. It gives $(x,f(x)+j)(w,f(w))=(xw,f(xw)+jf(w))=(0,0)$. We have $xw=0$ and $f(xw)+jf(w)=jf(w)=0.$ As $x\in Ann(w)=0$, we obtain $x=0$. Since $f$ is an isomorphism $jf(w)=0,$ we have by Lemma \ref{isomorphism}, $j\in Ann(f(w))=f(Ann(w))=f(0)=0,$ that is, $j=0$ and so $(x,f(x)+j)=(0,0)$. Thus, by hypothesis, $\exists (s,f(s)+j)\in S \bowtie^f J$ such that $(s,f(s)+j)(z,f(z))=(sz,f(sz)+jf(z))\in A \bowtie^f J.$ Hence, $sz\in A.$\\

 2) Note that $S\bowtie^f J \cap A\bowtie^fJ=\varnothing.$ Let $(w,f(w)+j_1)(z,f(z)+j_2)\in A \bowtie^f J$ for some $(w,f(w)+j_1)$ and $(z,f(z)+j_2)\in H_1 \bowtie^f J$, with $Ann((w,f(w)+j_1))=(0,0)$. We show that $Ann(w)=0.$ Since $J\subseteq zd(H_2)$ and $f$ is an epimorphism, there exist $0\neq z_1\in H_2$ and $x_1\in H_1$ such that $j_1z_1=0$ and $f(x_1)=z_1.$ Let $x\in Ann(w)$. We have $(x_1x,f(x_1x))(w,f(w)+j_1)=(x_1xw,f(x_1xw)+j_1f(x_1x))=(0,j_1z_1f(x))=(0,0)$. Hence $(x_1x, f(x_1x))\in Ann((w,f(w)+j_1))=(0,0)$ and $x_1x=0$. As $H_1$ is an integral domain and $x_1$ can not be zero (otherwise, $z_1=0$) we obtain $x=0$ and thus, $Ann(w)=0$. Since $wz\in A$ and $A$ is an $S-r-$ideal, there exists $s\in S$ such that $sz\in A.$ Fix $(s,f(s))\in S \bowtie^f J.$ We have $(s,f(s))(z,f(z)+j_2)=(sz, f(sz)+j_2f(s))\in A \bowtie^f J.$
\end{proof}

%\subsection{Trivial extension}

Let $H$ be a ring and $M$ be an $H-$module. Recall from \cite{Ander} that the set $$H\propto M=\{(r,m):r\in H,\; m\in M \}$$ 
is a ring defined by the following operation:\\
$(r_1,m_1)\cdot (r_2,m_2)=(r_1r_2, r_1m_2+r_2m_1)$. The multiplicative identity is $(1,0)$ and the zero of ring is $(0,0).$ We know from \cite{Ander} that $A\propto N $ is an ideal of $H\propto M$ if and only if $AM\subseteq N$, where $A$ is an ideal of $H$ and $N$ is a submodule of $M.$\\

The following result is proved in \cite{MH} under the condition $Z(M)=zd(H)$, where $Z(M)=\{r\in H: rx=0,\;for \;some\;0\neq x\in M\}$. We prove the result under the conditions where $M$ is torsion-free and $\bigcup_{a\in H}Ann(a)\subseteq Ann(M)$.

\begin{proposition}
  Suppose that $A\cap S=\varnothing.$  If $M$ is a torsion-free module and $\bigcup_{a\in H}Ann(a)\subseteq Ann(M)$, then the following statements are equivalent:\\
    1) $A$ is an $S-r-ideal$.\\
    2) $A \propto M$ is an $(S\propto 0)-r-ideal$ of $H\propto M$.\\
    3) $A \propto M$ is an $(S\propto M)-r-ideal$ of $H\propto M$.
\end{proposition}

\begin{proof}
    We can see that $S\propto M=\{(r,m): r\in S,\; m\in M\}$ is a m.c.s. of $H\propto M$. Because $0\notin S$, $(0,0) \notin S\propto M$, and since $1\in S$, we have $(1,0)\in S\propto  M$. Also, for all $(s_1,m_1),(s_2,m_2)\in S\propto M$ we have $(s_1,m_1)\cdot (s_2,m_2)= (s_1s_2,s_1m_2+s_2m_1) \in S\propto M$. Moreover, as $S\cap A=\varnothing$, we have $(S\propto 0) \cap (A \propto M)= \varnothing$. Now we can start the proof:\\
    $(1)\Rightarrow (2):$ Let $(w,m_1)(z,m_2)\in A\propto M$ for some $(w,m_1),(z,m_2)\in H\propto M$ with $Ann((w,m_1)=(0,0).$ Let $x\in Ann(w)$. By assumption, $xM=0$, in particular, we have $xm_1=0$. Note that $(x,0)(w,m_1)=(xw,xm_1)=(0,0)$ and hence $(x,0)\in Ann(w,m_1)=(0,0)$, that is, $x=0$ and so $Ann(w)=0$. As $wz\in A$ and $A$ is an $S-r-$ideal, there exists $s\in S$ such that $sz\in A$. Thus, $(s,0)\in S\propto 0$ and $(s,0)(z,m_2)=(sz,sm_2)\in A\propto M$.\\
    $(2)\Rightarrow (3):$ Since $S\propto  0\subseteq S\propto M$ and $(S\propto M) \cap (A\propto M)=\varnothing$, the result follows by Proposition \ref{Prop.23}.\\
    $(3) \Rightarrow (1):$ Let $wz\in A$ for some $w,z\in H$ with $Ann(w)=0$. Then $(w,0)(z,0)\in A\propto M$. Let $(r,x)\in Ann(w,0)$. We get $(r,x)(w,0)=(rw,xw)=(0,0)$, that is, $rw=0$ and $wx=0$. Since $Ann(w)=0$ and $M$ is torsion-free and $w\neq 0$, we have $r=0$ and $x=0$ and hence $Ann(w,0)=(0,0).$ By (3), there exists $(s,m)\in S\propto M$ such that $(s,m)(z,0)=(sb,bm) \in A\propto M.$ Thus, $sz\in A.$
\end{proof}

\section{$S-r-$Ideals in Polynomial Rings}\label{sec4}

Suppose that $f=\sum_{i=0}^n a_ix^i \in H[x]$, where $a_i'$s are the coefficients. Let $C(f)$ denote the set of all coefficients of $f$ and $c(f)$ denote the ideal of $H$ that is generated by $C(f)$. If $A$ is an ideal of $H$, then $A[x]=\{f\in H[x]:\; C(f)\subseteq A\}$ is an ideal of $H[x]$ \cite{M}. 
%If we take specially $f=\sum_{i=0}^{\infty} f_ix^i$, then $c(f)$ is the sequence ${f_n}_{n\in \mathbb{N}}.$
 Recall that if every finitely generated ideal $A\subseteq zd(H)$ has non zero annihilator, then $H$ is said to satisfy the Property A and if for every finitely generated ideal $A$ of $H$ there exists an element $z\in H$ with $Ann(A)=Ann(z)$, then $H$ is said to satisfy the annihilator condition (a.c.) \cite{Huck}.
 %and if for each finitely generated ideal $I$ of $R$ there exists an element $b\in I$ with $Ann(I)=Ann(b)$, then it satisfies strongly annihilator condition (s.a.c.) by \cite{Huck}, \cite{HJ}.

%\begin{example}
   % Since $\mathbb{Z}[x]$ is an integral domain, by \cite[Prop. 2.8]{M}, the ideal $0\neq A[x]$ can not be an $r-ideal$ of $\mathbb{Z}[x]$, where $A$ is an ideal of $\mathbb{Z}.$ But if we choose $S=\mathbb{Z}\setminus \{0\},\; A[x]$ is an $S-r-ideal$ of $\mathbb{Z}[x]$.
%\end{example}

\begin{theorem} \label{Poly.Ring.1}
    Suppose that $H$ is a ring, $A$ is an ideal of $H$ and $S\subseteq reg(H)$ is a m.c.s. of $H$. The following are equivalent:\\
    1) The ring $H$ satisfies property A.\\
    2) For every ideal $A$ of $H$, $A$ is an $S-r-ideal$ in the ring $H$ if and only if $A[x]$ is an $S-r-ideal$ in the ring $H[x]$.
\end{theorem}

\begin{proof}
    $(1) \Rightarrow (2):$ Note that $S\cap A[x]=\varnothing$ if and only if $S\cap A=\varnothing.$ Assume that $A$ is an $S-r-ideal$ in $H,\; w,z \in H[x]$ and $zw\in A[x]$ with $Ann_{H[x]}(z)=0$. Then, by [\cite{AM}, Proposition 3.5], we obtain that $c(z)\nsubseteq zd(H)$. Thus, $\exists r\in c(z)$ such that $Ann_H(r)=0$. It is clear that $C(wz)\subseteq A$ and it gives $c(wz)\subseteq A.$ By \cite[Theorem 28.1]{Gilmer}, we get $c(z)^{m+1}c(w)=c(z)^mc(wz),$ where $m$ is the degree of $w.$ This gives $c(z)^{m+1}c(w)\subseteq A.$ As $r^{m+1}\in c(z)^{m+1}$, we have that $r^{m+1}c(w)\subseteq A.$ On the other hand, we have $Ann_H(r^{m+1})=0$. Now, we conclude that by Proposition [\cite{MH}, Theorem 2.1]  $\exists s \in S,\; sc(w)\subseteq A.$ Thus, $sw \in A[x].$
    Conversely let $s\in S$ satisfy the $S-r-ideal$ property for $A[x]$. Let $wz\in A$ and $Ann_H(w)=0$. Since $wzx^0\in A[x]$ and $Ann_{H[x]}(wx^0)=0,$ we get $sz \in A[x]$ and so $sz\in A.$\\
    $(2)\Rightarrow (1):$ Assume that Property A does not hold for $H$. By \cite[Proposition 3.5]{AM}, $\exists w\in H[x]$ with $Ann_{H[x]}(w)=0$ and $A=c(w)\subseteq zd(H)$. Then by Remark \ref{Remark1}, we can find an ideal $B$ with $B\cap S=\varnothing$ (since $S\subseteq reg(H)$) such that $B$ is a prime $S-r-$ideal and $A=c(w)\subseteq B$. Therefore, $w\in B[x],$ while $w$ is a regular element. Thus, $B[x]$ is not an $S-r-ideal$, a contradiction.   
\end{proof}

Let $H$ be a ring. $H$ is called to have the finite annihilator condition (briefly, $f.a.c.$) if for each finite subset  $A$ of $H$ there is $w\in A$ such that $Ann(A)=Ann(w)$, see \cite{M}.

%It is clear that if $R$ satisfies the $f.a.c.$, then it satisfies the $s.a.c.,$ and so satisfies the a.c., $R$ may satisfies property A but it may not satisfy $a.c.$ and also $f.a.c.$ see [[2],Ex.4.1]

\begin{theorem} \label{Poly.Ring.2}
    
    Suppose that $H$ is a ring that satisfies the $f.a.c.$ and let $S\subseteq H$ be a m.c.s. Then $A$ is an $S-r-ideal$ in the ring $H$ if and only if $A[x]$ is an $S-r-ideal$ in the ring $H[x]$.
    %(resp. $R[[x]])$.
\end{theorem}

\begin{proof}
    Assume that $w,z\in H[x]$ and $wz\in A[x]$ with $Ann_{H[x]}(w)=0$. Hence $Ann_H(C(w))=0$. By hypothesis, there exists $y\in C(w)$ such that $Ann_H(C(w))=Ann_H(y)$. Thus, $Ann_H(y)=0$. It is easy to check that since $y\in C(w)$ and $wz\in A[x]$, we have $y C(z)\subseteq A$ and so $yc(z)\subseteq A$. As $A$ is an $S-r-ideal$ in $H,$ we conclude that $\exists s\in S, \; sc(z)\subseteq A.$ This gives that $sz\in A[x]$, that is, $A[x]$ is an $S-r-ideal$ in $H[x]$. The converse is clear. 
    %Also the case of $I[[x]]$ can be done similiary under the condition of $c.a.c.$.
\end{proof}

\begin{corollary}
 Let $S\subseteq reg(H)$ be a finite m.c.s. of $H$. Assume that $H$ is an $S-uz-$ring and $A\cap S=\emptyset$ for every ideal $A$ of $H$. The followings are equivalent:\\
    (1) $H$ satisfies Property A.\\
    (2) $A[x]$ is an $S-r-$ideal in $H[x]$, for every ideal $A$ of $H$.
\end{corollary}

\begin{proof}
Follow by Theorem \ref{Poly.Ring.1} and Proposition \ref{S-uz}.
\end{proof}

%%%%%%%%%%%%%%

\section{A Conflict of Interest Statement}
The corresponding author states that there is no conflict of interest.

 \end{document}